\documentclass[10pt]{amsart}
\usepackage{amssymb,amsmath,amsthm,amsfonts,amsopn,url}
\usepackage{amscd,amssymb,amsopn,amsmath,amsthm,graphics,amsfonts,enumerate,verbatim,calc}
\usepackage[mathcal]{euscript}
\usepackage[all]{xy}
\theoremstyle{plain}
\newtheorem{thm}{Theorem}[section]
\newtheorem*{mt*}{Main Theorem}

\newtheorem{prop}[thm]{Proposition}
\newtheorem{lemma}[thm]{Lemma}
\newtheorem{cor}{Corollary}
\newtheorem{rem}{Remark}
\newtheorem{example}{Example}
\theoremstyle{definition}

\newcommand{\ideal}[1]{\mathfrak{#1}}
\newcommand{\m}{\ideal{m}}
\newcommand{\n}{\ideal{n}}
\newcommand{\p}{\ideal{p}}
\newcommand{\q}{\ideal{q}}

\newcommand{\func}[1]{\mathrm{#1} \,}

\newcommand{\Tor}{\func{Tor}}
\newcommand{\Ext}{\func{Ext}}
\newcommand{\Ass}{\func{Ass}}
\newcommand{\V}{\func{V}}
\newcommand{\T}{\func{T}}
\newcommand{\Z}{\func{Z}}
\newcommand{\Supp}{\func{Supp}}
\newcommand{\Spec}{\func{Spec}}

\newcommand{\ZZ}{{\mathbb Z}}
\newcommand{\CC}{{\mathbb C}}

\title[]{Behaviour of Finiteness of the Set of Associated Primes under Ring Extensions}

\author[]{Rajsekhar Bhattacharyya}

\address{Dinabandhu Andrews College, Garia, Kolkata 700084, India}

\email{rbhattacharyya@gmail.com}

\thanks{}


\begin{document}

\begin{abstract}
We study the behaviour of the finiteness of the set of associated primes of local cohomology modules, more generally of Lyubeznik functors, under various ring extensions. At first, we review the results for flat and faithfully flat extensions and we present new applications of them. Then, we focus how the finiteness property of the set of associated primes of local cohomology modules and Lyubeznik functors is transferred from extended ring to the base ring of pure and cyclically pure ring extensions. We show that finiteness property can be transferred from a ring to its pure local subring and this extends the result of Theorem 1.1 of \cite{Nu}. Further, we observed that under mild conditions on the rings, finiteness property comes down from cyclically pure ring extensions to its local base ring. In particular, we observe that the set of associated primes of Lyubeznik functors of a cyclically pure local subring (which turns out to be Cohen-Macaulay) of equicharacteristic or unramified regular local ring, is finite. There is an appendix on behaviour of the Bass numbers under pure and cyclically pure ring extensions.
\end{abstract}

\maketitle
\section{introduction}
Let $R$ be a Noetherian ring and $M$ be a module over it. For an ideal $I\subset R$ and for some integer $i\geq 0$, consider local cohomology module $H^i_I (M)$ with support in the ideal $I$. In the fourth problem of \cite{Hu}, it is asked that whether local cohomology modules of Noetherian rings have finitely many associated prime ideals. There are examples given in \cite{Si}, \cite{Ka}, and \cite{SS}, which show that the set of associated primes of $i$-th local cohomology module $H^i_I (R)$ of a Noetherian ring $R$ can be infinite. However, there are several important cases where we have the finiteness of the set of associated primes of local cohomology modules and we list them below: (1) Regular rings of prime characteristic \cite{HS}, (2) regular local and affine rings of characteristic zero \cite{Ly1}, (3) unramified regular local rings of mixed characteristic \cite{Ly2} and (4) smooth algebra over $\ZZ$ \cite{BBLSZ}. These results support Lyubeznik conjecture, (\cite{Ly1}, Remark 3.7): 

\textbf{Conjecture:} Let $R$ be a regular ring and $I\subset R$ be its ideal, then for each $i\geq 0$, $i$-th local cohomology module $H^i_I (R)$ has finitely many associated prime ideals.

It is well known that for any Noetherian ring $R$ of dimension $d$ and for any $R$-module $M$, $\Ass_R H^d_I (M)$ is always finite, see \cite{Mar}. Moreover, if the ring becomes local then for any $R$-module $M$, we also have the finiteness of $\Ass_R H^{d-1}_I (M)$, see Corollary 2.4 of \cite{Mar}. Recently, in \cite{Pu2}, it is shown that for an excellent regular ring $R$ of dimension $d$ containing a field of characteristic zero, $\Ass_R H^{d-1}_I (R)$ is finite, for $i\geq 0$ and for every ideal $I\subset R$. 

The conjecture is open for the ramified regular local ring, in fact, more specifically, it is open only for the primes that contain $p$ (see \cite{Nu3}).

In this paper, we study how the finiteness of the set of associated primes of local cohomology modules transfers from extended ring to the base ring and vice versa, under various ring extensions. More generally, whenever it is possible, we extend our study to Lyubeznik functor, which is introduced by Lyubeznik in \cite{Ly1}. Here, we account for a brief description of it: Let $\Z$ be a closed subset of $\Spec R$ and $M$ be an $R$-module. We set $H^i_{\Z} (M)$ as the $i$-th local cohomology module of $M$ with support in $\Z$. We notice that $H^i_{\Z} (M)=H^i_I (M)$, for $\Z=\V(I)=\{P\in\Spec R: I\subset P\}$. For any two closed subsets of $\Spec R$, $\Z_1\subset \Z_2$, there is a long exact sequence of functors
$$\ldots\to H^i_{\Z_1}\to H^i_{\Z_2}\to H^i_{\Z_1/\Z_2}\to \ldots$$
We set $\T =\T_1\circ \dots\circ\T_t$, where every functor $\T_j$ is either $H^i_{\Z}$ for some closed subset $\Z$ of $\Spec R$ or the kernel or image (or cokernel) of some map in the above long exact sequence. $\T$ is known as Lyubeznik functor.

In section 2, we review the situations when the ring extensions are flat and faithfully flat. In section 3, we observe applications of these cases: Firstly, in Theorem 3.2 we show that for any smooth algebra over Noetherian semi-local domain of characteristic zero, where each of its maximal ideal is generated by a prime in $\ZZ$, each local cohomology module has finitely many associated prime ideals. This generalizes the result of the part of the Theorem 4.1 of \cite{BBLSZ}. Secondly, in Theorem 3.3, we study how the finiteness property can come down from localized ring to base ring. In \cite{Nu2}, finiteness of the set of associated primes of local cohomology for polynomial ring or for power series ring over one dimensional Noetherian ring is studied. Lastly, in Theorem 3.4, we extend the part of the result of Theorem 1.2 of \cite{Nu2}, to polynomial and power series ring over two dimensional Noetherian local ring. 

`Cyclically pure' condition on a ring extension is the weakest condition among the conditions namely, `faithfully flat', `split' and `pure'. We have `faithfully flat' $\Rightarrow$ `pure', `split' $\Rightarrow$ `pure' where `pure' $\Rightarrow$ `cyclically pure'. 
In section 4, we observe the behaviour of finiteness of the set of associated primes of local cohomology modules as well as Lyubeznik functor under pure and cyclically pure ring extensions. In Theorem 1.1 of \cite{Nu}, it is proved that finiteness property of the set of associated primes of Lyubeznik functor $\T$ can be transferred from the ring to its subring which is a direct summand of it. In Theorem 4.1, we show that finiteness of the set of associated primes of $\T$ can be transferred from the ring to its pure local subring and this extends the result of Theorem 1.1 of \cite{Nu}. In Lemma 4.2, we show that after a base change by a complete ring, cyclically pure extension remains cyclically pure. Using Lemma 4.2 and under mild conditions on the rings, we find that finiteness of the set of associated primes of Lyubeznik functor $\T$ comes down from extension ring to its cyclically pure subring. We state this result in Theorem 4.3. As a consequence, we observe that the set of associated primes of $\T$ of a cyclically pure subring (which turns out to be a Cohen-Macaulay ring) of equicharacteristic or unramified regular local ring, is finite, see Corollary 7 and Remark 6. Lastly, there is an appendix on behaviour of the Bass numbers under pure and cyclically pure ring extensions. 

Throughout the papers all the rings are commutative Noetherian ring with unity. For basic results and for unexplained terms, we refer \cite{BH}, \cite{CAofM} and \cite{CRTofM}.

\section{a review on basic results}

In this section, we review some of the known results regarding the behaviour of finiteness of the set of associated primes of local cohomology modules under flat and faithfully flat ring extensions. 

For flat extension of Noetherian rings, finiteness property of set of associated primes can be transferred from base ring to extended ring, but only under special situation we can do the reverse. As for example, for faithfully flat extension we can always do the reverse. The results of flat and faithfully flat extensions are given in (1) of Proposition 2.1 and in Corollary 1. More specifically, in (2) of Proposition 2.1, we get the result for the reverse in the flat extension and from which Corollary 1 follows. The results of (2) of Proposition 2.1 extend (4) of Theorem 3.3.1 of \cite{Gl}, from localization to an arbitrary flat extension.  

\begin{prop}
Suppose $R\rightarrow S$ be a ring homomorphism of Noetherian rings. Let $S$ be flat over $R$.\newline 
(1) For $R$-module $E$, if $\Ass_R E$ is finite then $\Ass_S (E\otimes_R S)$ is also finite.\newline 
For an $R$-module $M$, for every ideal $I\subset R$ and for $i\geq 0$, $\Ass_{R} H^i_I(R)$ is a finite set implies that $\Ass_{S} H^i_{IS}(S)$ is a finite set.\newline
More generally, for Lyubeznik functor $\T$, if $\Ass_{R} \T(R)$ is a finite set then $\Ass_{S} \T(S)$ is a also finite set.\newline
 
(2) Converse will be true, if in addition to we assume that only for finitely many primes $\p\in \Ass_{R} E$, $\p S$ is not a proper ideal of $S$. For the local cohomology modules and for Lyubeznik functors, converse will be true if in addition to we assume that only for finitely many primes $\p\in \Ass_{R} H^i_I(R) $ and $\p\in \Ass_{R} \T(R)$, $\p S$ is not a proper ideal of $S$.
\end{prop}

\begin{cor}
Let $S$ be faithfully flat over $R$.\newline 
(1) For $R$-module $E$, $\Ass_{R} E$ is finite if and only if $\Ass_{S}(E\otimes_{R} S)$ is finite.\newline
(2) In particular, for the local cohomology modules of an arbitrary $R$-module $M$, for every ideal $I\subset R$ and for $i\geq 0$, $\Ass_{S} H^i_{IS}(M\otimes_R S)$ is a finite set if and only if $\Ass_{R} H^i_I(M)$ is a finite set.\newline
(3) More generally, for Lyubeznik functor $\T$, $\Ass_{S} \T(S)$ is a finite set if and only if $\Ass_{R} \T(R)$ is a finite set.\newline
\end{cor}

\begin{rem}
Both, for Proposition 2.1 and Corollary 1, assertions (1) and (2) are well known. For (3), we observe that $\T(S)= \T(R)\otimes_R S$ (see Lemma 3.1 of \cite{Ly1}), and result follows from first assertion. 
\end{rem}

For local cohomology modules we have the following lemma. The results are well known result, but for the sake of completeness we present the proof.

\begin{lemma}
Consider an $R$-module $M$ and let $I$ be an ideal of Noetherian ring $R$. For $i\geq 0$, let $H^i_{I}(M)$ be its $i$-th local cohomology module. Then, $\Ass_{R}H^i_{I}(M)\subset \Supp_{R}H^i_{I}(M)\subset V(I)$. 
\end{lemma}

\begin{proof}
Since $\Ass_{R}H^i_{I}(M)\subset \Supp_{R}H^i_{I}(M)$, it is sufficient to show $\Supp_{R}H^i_{I}(M)\subset V(I)$. For $\p\in \Spec R$, assume that $(H^i_I (M))_p\neq 0$, then $\p$ should contain $I$, otherwise, we get some element $x\in I$ outside $\p$ which is unit in $R_{\p}$. Since, every element of $H^i_I (M)$ is annihilated by some power of $I$, every element of $(H^i_I (M))_p$ is annihilated by some power of $x$. This implies that $(H^i_I (M))_p = 0$. Thus $\Ass_{R}H^i_{I}(M)\subset \Supp_{R}H^i_{I}(M)\subset V(I)$. 
\end{proof}

We focus on the results of special situation of Proposition 2.1, when $S= W^{-1}R$ for a multiplicatively closed set $W$ of $R$ and we summarise all these well-known results in the following proposition.  

\begin{prop}
Let $R$ be a Noetherian ring, $\p,\q\in\Spec R$ and $W$ be a multiplicatively closed set. Then for an arbitrary $R$-module $M$, we have the following:\newline
(1) For every $i\geq 0$ and for every ideal $I\subset R$, $\p W^{-1}R\in \Ass_{W^{-1}R} H^i_{IW^{-1}R} (W^{-1}M)$ if and only if $\p\in \Ass_R H^i_I (M)$ and $\p\cap W=\phi$.\newline 
As a consequence, $\Ass_RH^i_I(M)$ is finite implies that $\Ass_{W^{-1}R}H^i_{IW^{-1}R}(W^{-1}M)$ is finite.\newline
(2) For every $i\geq 0$, if only finite number of primes of $\Ass_RH^i_I(M)$ intersects with $W$, then finiteness of $\Ass_{W^{-1}R}H^i_{IW^{-1}R}(W^{-1}M)$ implies the finiteness of $\Ass_R H^i_I(M)$.\newline 
In particular, for every $i\geq 0$, if only finite number of primes of $\V(I)$ intersects with $W$, then finiteness of $\Ass_{W^{-1}R}H^i_{IW^{-1}R}(W^{-1}M)$ implies the finiteness of $\Ass_R H^i_I(M)$.\newline
(3) Let $W=\{a^n: a\in R, n\in {\ZZ}_+\}$. For every $i\geq 0$, if $\Ass_R H^i_I(M)$ is contained in the open subset $D(a)=\Spec R-V(a)$, then finiteness of $\Ass_{W^{-1}R}H^i_{IW^{-1}R}(W^{-1}M)$ implies the finiteness of $\Ass_R H^i_I(M)$.\newline 
In particular, for every $i\geq 0$, if $\V(I +(a)R)$ is finite, then finiteness of $\Ass_{W^{-1}R}H^i_{IW^{-1}R}(W^{-1}M)$ implies the finiteness of $\Ass_R H^i_I(M)$.
\end{prop}

In (1) of Proposition 2.3, we observe how finiteness property of associated primes is transferred from $R$ to $W^{-1}R$, while in (2) and (3) of Proposition 2.3, we have partial results that how finiteness property of associated primes can be transferred from $W^{-1}R$ to $R$. For a more comprehensive situation, see Theorem 3.3.

\section{behaviour of finiteness of the set of associated primes under flat and faithfully flat ring extensions}

In this section, we study three applications of the basic results as discussed in the previous section. 

\subsection{Generalization of the result of the part of the Theorem 4.1 of \cite{BBLSZ}}

At first, we recall the definition of regular algebra, see (\cite{CAofM}, page 249). Here, we call them as smooth algebra. We observe the following result for a smooth algebra.

\begin{lemma}
If $B$ is a smooth $A$-algebra, then for any multiplicatively closed set $W$ of $A$, $W^{-1}B$ is a smooth $W^{-1}A$-algebra.
\end{lemma}

\begin{proof}
Clearly $W^{-1}B$ is flat over $W^{-1}A$. Consider $P\in \Spec W^{-1}A$ with $P=\p W^{-1}A$ for some $\p\in \Spec A$. Now $\kappa(\p W^{-1}A)=(W^{-1}A)_{\p W^{-1}A}/(\p W^{-1}A)(W^{-1}A)_{\p W^{-1}A}= A_{\p}/\p A_{\p}=\kappa(\p)$. Thus for any finite extension $L$ of $\kappa(\p W^{-1}A)$, $W^{-1}B\otimes_{W^{-1}A} L= (W^{-1}B\otimes_{W^{-1}A} \kappa(\p W^{-1}A))\otimes_{\kappa(\p W^{-1}A)} L= ((B\otimes_A W^{-1}A)\otimes_{W^{-1}A} \kappa(\p W^{-1}A))\otimes_{\kappa(\p W^{-1}A)} L= B\otimes_A \kappa(\p)\otimes_{\kappa(\p)} L= B\otimes_A L$. Since $B\otimes_A L$ is regular, we conclude.
\end{proof}

For the rest of this subsection we adopt the following terminology: We say a Noetherian ring $R$ satisfies `Finiteness condition of associated primes' if for every ideal $I\subset R$ and for every $i\geq 0$, $\Ass_R H^i_I (R)$ is finite. For primes $a_1,\ldots,a_n$ in $\ZZ$, let $\ZZ_{a_1,\dots,a_n}$ be the ring where every prime of $\ZZ$ is inverted except the prime ideals $a_1,\ldots,a_n$. Consider a Noetherian semi-local domain $V_{a_1,\ldots,a_n}$ of characteristic zero with $n$ maximal ideals where each of them is generated by each of the prime integers $a_1,\ldots,a_n$. Clearly $\ZZ_{a_1,\dots,a_n}$ sits inside $V_{a_1,\ldots,a_n}$. This is a semi-local Noetherian domain of mixed characteristic where every maximal ideal is generated by some prime integer of $\ZZ$. Certainly, this semi-local ring is a generalization of $p$-ring as defined in \cite{CRTofM}. 

In the following theorem, we apply the results of Proposition 2.3 to show that any smooth algebra $R$ over $V_{a_1,\ldots,a_n}$ satisfies `Finiteness condition of associated primes'. This generalizes the part of the result of Theorem 4.1 of \cite{BBLSZ}.

\begin{thm}
(1) Let $a_1,\ldots,a_n$ be prime integers in $\ZZ$, Let $R$ be a finitely generated flat $V_{a_1,\ldots,a_n}$-algebra containing $V_{a_1,\ldots,a_n}$ as a subring of it. For any prime $a\in \ZZ$, if every finitely generated flat algebra over $V_a$ satisfies `Finiteness condition of associated primes', then for every $i\geq 0$ and for every $I\subset R$, we have $\Ass_R H^i_I(R)$ is finite. 

(2) In particular, let $a_1,\ldots,a_n$ be prime integers in $\ZZ$, Let $R$ be a smooth $V_{a_1,\ldots,a_n}$-algebra containing $V_{a_1,\ldots,a_n}$ as a subring of it. Then for every $i\geq 0$ and for every $I\subset R$, we have $\Ass_R H^i_I(R)$ is finite. 
\end{thm}

\begin{proof}
(1) We proceed by induction on $n$. When $n=1$, the result follows from hypothesis. Now we consider the case $n=2$. Set $a_1 =a$ and $a_2= b$. According to the hypothesis, in $R$, $a$ and $b$ are coprimes. For any $I\subset R$, if it contains $b$, we always have $I+ (a)R= R$. Thus $\V(I)\cap \V((a))= \V(I+(a))=\phi$. Thus $I$ can not contain both $a$ and $b$. So, we assume that $b\in I$. On the other hand, form Lemma 2.2 we get, $\Ass_{R}H^i_{I} (R)\subset \Supp_{R} H^i_{I}(R)\subset V(I)$. This implies that no prime of $\Ass_{R} H^i_I (R)$ intersects with $W_a=\{a^n: n\in {\ZZ}_+\}$. 

Now $V_b$ is localization of $V_{a, b}$ for the multiplicatively closed set $W_a$. Let $R_a$ be the ring of localization for the multiplicatively closed set $W_a$. Since $R$ is a finitely generated flat $V_{a, b}$-algebra, by above Lemma $R_a$ is also a finitely generated flat $V_b$-algebra. From hypothesis we find that $\Ass_{R_a} H^i_J(R_a)$ is a finite set for every $i\geq 0$ and for every ideal $J\subset R_a$. Thus from (3) of above Proposition 2.3, for every $i\geq 0$ and for every $I\subset R$ which contains $b$, we always have $\Ass_R H^i_I(R)$ is finite. Now from the symmetry in the above argument, the same holds when $a\in I$.  

Finally, assume neither $a$ nor $b$ is in $I$. Consider the multiplicatively closed set $W_a$ and $W_b$. If there are finitely many primes of $\Ass_R H^i_I(R)$ intersects $W_a$ or $W_b$, then using (2) of Proposition 2.3, we find $\Ass_R H^i_I(R)$ is finite. Otherwise, each of the elements $a$ and $b$ are in infinitely many primes from $\Ass_R H^i_I (R)$. Set $A_a=\{{\p}\in \Ass_R H^i_I(R): a\notin \p \}$, $A_b=\{{\p}\in \Ass_R H^i_I(R): b\notin \p \}$ and $A_{a,b}=\{{\p}\in \Ass_R H^i_I(R): a\notin \p, b\notin \p\}$. Here, $A_{a,b}=A_a \cap A_b$ and moreover no prime can contain both $a$ and $b$, since $a$ and $b$ are coprimes. Thus we find $\Ass_R H^i_I(R)= A_a \cup A_b$. 

Now we observe the following: For $a\in R$, using (1) of Proposition 2.3 we get that, $\p\in A_a$ if and only if $\p R_a\in \Ass_{R_a} H^i_{IR_a} (R_a)$. From hypothesis $\Ass_{R_a} H^i_{IR_a} (R_a)$ is finite. Similar is true for $b$. Thus $\Ass_R H^i_I (R)$ is finite and this concludes the proof of case $n=2$.

Assume the result is true for $n-1$ i.e. any finitely generated flat algebra over $V_{{a_1},\ldots,{a_{n-1}}}$ satisfies `Finiteness condition of associated primes'. Now consider a finitely generated flat algebra $R$ over $V_{{a_1},\ldots,{a_{n}}}$. Set $a_1= a$, $a_2= b$ and $V_{{a_3},{a_3}\ldots,{a_{n}}}= U$. Clearly, $V_{{a_1},\ldots,{a_{n}}}= U_{a, b}$. Now $U_b$ is localization of $U_{a, b}$ for the multiplicatively closed set $W_a$. Let $R_a$ be the ring of localization for the multiplicatively closed set $W_a$. Since $R$ is a finitely generated flat $U_{a, b}$-algebra, by above Lemma $R_a$ is also a finitely generated flat $U_b$-algebra. From inductive hypothesis, we find that $\Ass_{R_a} H^i_J(R_a)$ is a finite set for every $i\geq 0$ and for every ideal $J\subset R_a$. Thus, from (3) of above Proposition 2.3, for every $i\geq 0$ and for every $I\subset R$ which contains $b$, we always have $\Ass_R H^i_I(R)$ is finite. Now from the symmetry in the above argument, the same holds when $a\in I$.  

Finally, assume neither $a$ nor $b$ is in $I$. In this case, we can argue in the similar way to that of $n= 2$ case. 
Thus we get that $\Ass_R H^i_I (R)$ is finite and this proves the result for $n$. Thus inductively, we conclude the proof.  

(2) From Theorem 4.1 of \cite{BBLSZ} we know that every smooth algebra over $V_a$ satisfies `Finiteness condition of associated primes'. Now, using Lemma 3.1 the result follows from (1). 
\end{proof}

\subsection{Behaviour of finiteness of associated primes; when we come down from localized ring to base ring}
 
In Proposition 2.1 we observe that, for flat extension of Noetherian rings, finiteness property of set of associated primes can be transferred from base ring to the extended ring. We can do the reverse for faithfully flat situation. Localization is never faithfully flat unless we localize the ring trivially, i.e. by the set of units. Due to the flatness, in (1) of Proposition 2.3, we observe that `Finiteness condition of associated primes' goes up from base ring to its localization. But, we have only partial results that how finiteness property can come down from localized ring to base ring, see (2) and (3) of Proposition 2.3. 

In the following theorem i.e. in Theorem 3.3 below, we study that how finiteness property can come down from localized ring to base ring. 
 
\begin{thm}
Let $R$ be a Noetherian ring of dimension $d$, and $M$ an $R$-module. Suppose that $Ass_R M_f$ is finite for every not unit $f\in R$. Then, $Ass_R M$ is finite.\newline 
In particular, if $Ass_R H^i_I (R_f)$ is finite for every non unit $f\in R$ and every ideal $I\subset R$, then, $Ass_R H^i_I (R)$ is finite.
\end{thm}

\begin{proof}
Fix a maximal ideal $m = (f_1\ldots,f_n)\subset R$. Let $D_f = \Spec (R)-\V (f)$, where $\V (f) = \{\p\in Spec(R):f\in \p\}$. We know that $\Spec (R) = D_{f_1}\cup\ldots\cup D_{f_n}\cup\{m\}$. Then,
$$\Ass_R M = (D_{f_1}\cap \Ass_R M)\cup\ldots\cup (D_{f_n}\cap \Ass_R M)\cup (\{m\}\cap\Ass_R M)$$
$$= \Ass_R M_{f_1}\cup\ldots\cup\Ass_R M_{f_1}\cup (\{m\}\cap\Ass_R M)$$
$$\subset \Ass_R M_{f_1}\cup\ldots\cup\Ass_R M_{f_1}\cup\{m\}$$
Hence, $\Ass_R M$ is finite.

Now the second assertion is immediate from the first assertion.
\end{proof}

We conclude the section with the following corollaries. This again shows that how `finiteness of the set of associated primes' comes down from localized ring to the base ring.

\begin{cor}
Let $R$ be a Noetherian local domain of dimension $d$, which contains a field of characteristic $p>0$. Let for every prime $\p$ of the punctured spectrum of $R$, $R_{\p}$ is regular, then for an ideal $I\subset R$ and for $i\geq 0$, $\Ass_R H^i_I (R)$ is finite.
\end{cor}

\begin{proof}
For any non-unit $a\in R$, for every $\p$ such that $a\notin\p$, we have $R_{\p}$ is regular. So $R_{W_a}$ is regular. Using the result of \cite{HS} the assertions follows.
\end{proof}

\begin{cor}
Let $R$ be an excellent Noetherian local domain of dimension $d$, which contains a field of characteristic zero, then for an ideal $I\subset R$ and for $i= d-1, d$, $\Ass_R H^i_I (R)$ is finite. 
\end{cor}

\begin{proof}
For any non-unit $a\in R$, here, again $R_{W_a}$ is an excellent regular ring. Using the results of \cite{Pu2} and \cite{Mar} the assertions follows.
\end{proof}

\subsection{Extension of part of the result of Theorem 1.2 of \cite{Nu2}}

In \cite{Nu2}, finiteness of the set of associated primes of local cohomology for polynomial ring or for power series ring over one dimensional Noetherian ring is studied. There, in Theorem 1.2, we observe a partial result regarding this fact. Here, we extend that that result in the following theorem.

\begin{thm}
Let $(R, \m)$ be a two dimensional Noetherian local ring and let $S$ be a polynomial ring over $R$. For an ideal $J\subset S$ and for $i\geq 0$, consider the local cohomology module $H^i_J(S)$. Then, the subset of $\Ass_S H^i_J(S)$, whose every element contains only one fixed parameter, is finite. 
\end{thm}

\begin{proof}
Take a system of parameters $\sigma, \tau$ and consider the $\m$-primary ideal $(\sigma, \tau)R$. Inverting $\sigma$, we get one dimensional ring $R_{\sigma}$. Now $R/\tau R$ is local with maximal ideal $\m/\tau R$. Set $R/\tau R=\bar{R}$ and $\m/\tau R=\bar{\m}$. Clearly, $\sigma \bar{R}$ is $\bar{\m}$-primary ideal and thus $\dim \bar{R}_{\bar{\sigma}}= \dim \bar{R}_{\sigma}= \dim R_{\sigma}/\tau R_{\sigma}= 0$.

For $i\geq 0$ and for an ideal $J$ of $S$, consider $H^i_J(S)$. Let $B\subset \Ass_S H^i_J(S)$ contains only one of the parameter $\tau$ (say). Consider $S_{\sigma}$, which is a polynomial ring over $R_{\sigma}$ and for such a ring, using (1) of Proposition 2.3 we have, $\p S_{\sigma}\in \Ass_{S_{\sigma}} H^i_{JS_{\sigma}}(S_{\sigma})$ if and only if $\p\in \Ass_S H^i_J(S)$ and $\sigma\notin \p$. Since $\p\in B$ contains only one parameter $\tau$, it should not contain $\sigma$ and hence $B\subset \Ass_{S_{\sigma}} H^i_{JS_{\sigma}}(S_{\sigma})$. 
Since, $\dim R_{\sigma}/\tau R_{\sigma}= 0$, from Proposition 3.6 of \cite{Nu2} we find that the subset $B$ of $\Ass_{W_{\sigma}} H^i_{JS_{\sigma}}(S_{\sigma})$ is finite. Thus we conclude. 
\end{proof}

\section{behaviour of finiteness of the set of associated primes under pure and cyclically pure ring extensions}

In this section, we study the behaviour of the finiteness of the set of associated primes of Lyubeznik functor under pure and cyclically pure ring extensions.
At first, we review basic results for pure extensions. In Theorem 1.1 of \cite{Nu}, it is proved that finiteness property of the set of associated primes of Lyubeznik functor $\T$ can be transferred from the ring to its subring which is a direct summand of it. In Theorem 4.1, we show that finiteness property of the set of associated primes of $\T$ can be transferred from the ring to its pure local subring and this extends the result of Theorem 1.1 of \cite{Nu}. Then, in Lemma 4.2, we show that after a base change by a complete ring, cyclically pure extension remains cyclically pure. In Theorem 4.3, where using Lemma 4.2, we show that under mild conditions on the rings, the finiteness property of of the set of associated primes of Lyubeznik functor $\T$ can be transferred from the ring to its cyclically pure local subring. In particular, we observe that the set of associated primes of $\T$ of a cyclically pure local subring (which turns out to be Cohen-Macaulay) of equicharacteristic or unramified regular local ring, is finite, see Corollary 7 and Remark 6.

We recall the definitions of pure and cyclically pure ring extension: Let $A\rightarrow B$ be an injective ring homomorphism. For every $A$-module $M$, if $M\rightarrow M\otimes B$ is injective then we say $A$ is a pure subring of $B$ or the ring extension $A\rightarrow B$ is a pure ring extension. Moreover, if the map remains injective only after tensoring with every $A$-module of the form $A/I$ for some ideal $I\subset A$ or equivalently, if for every ideal $I\subset A$, $A/I\rightarrow B/IB$ is injective then we call the ring extension $A\rightarrow B$ as cyclically pure ring extension.

\begin{rem}
From definition, it is immediate that purity implies cyclic purity, but the converse is true when the ring is approximately Gorenstein, see \cite{Ho1}. We mention examples of approximately Gorenstein ring:

(a) For complete semilocal reduced ring, more generally locally excellent reduced ring,

(b) For normal domain.
\end{rem}

Here, we present examples of pure ring extensions:

\begin{example}
(a) If a ring extension is faithfully flat, then it is a pure extension.

(b) For a ring homomorphism $A\rightarrow B$, if $A$ is a direct summand of $B$, equivalently if the map splits then it is pure.

(c) Here we mention a special situation when $A$ is a direct summand of $B$: Consider a group of automorphism $G$ of the ring $B$ and let $A=B^G$. If there exists Reynolds operator for the ring extension $B^G\rightarrow B$, then it is a pure extension. For ring extension $B^G\rightarrow B$, Reynolds operator exists if, (i) $G$ is linearly reductive algebraic group (see `Main Theorem' in \cite{HR}) or, (ii) $G$ is finite and $G$ has an inverse in $B$.  

\end{example}

We also observe the following examples of cyclically pure ring extensions:

\begin{example}
(a) Every pure extensions are cyclically pure (see \cite{Ho1} or Remark 1, for the situations when they are equivalent).

(b) In \cite{Ho1}, we also have examples of cyclically pure subrings of small dimensions. In section 3 and Example 5.4 of \cite{Ho1}, we can observe cyclically pure local subrings of dimension zero and one (see also reference [2] and [4] there). 

(c) Let $R$ be a $\CC$-affine local domain. Then, there exists a big Cohen-Macaulay algebra $B(R)$ of $R$, see \cite{Sc1}. Observe the definition of $B(R)$-regular ring, see  5.1 Definition of \cite{Sc1} (as for example, regular rings are $B(R)$-regular ring). It follows immediate from the definition that $R$ is $B(R)$-regular if and only if $R\rightarrow B(R)$ is cyclically pure. 

(d) Consider the situation of above example. Let $S$ be a subring (possibly Noetherian) of $B(R)$ such that $R\rightarrow S$ is an inclusion map. Since $I\subset IS\cap R\subset IB(R)\cap R = I$, we get that $R$ is cyclically pure in $S$. Thus, if $R$ is $B(R)$-regular then for any subring $S$ (possibly Noetherian) of $B(R)$ containing $R$, $R$ is cyclically pure in $S$. 

(e) In (\cite{Sc2}, Theorem 2.2), we have another `sufficient' condition for a Noetherian local ring to be a cyclically pure subring of an arbitrary ring: Let $(R,m)$ be a Noetherian local ring with residue field $k$ and let $S$ be an arbitrary $R$ algebra. If $\Tor^R_1 (S, k) = 0$ and $mS \neq S$, then $R\rightarrow S$ is cyclically pure. 
\end{example}

We note the following observation.

\begin{rem}
Let $A\rightarrow B$ be a homomorphism of Noetherian rings. Let $M$ be a finitely generated $B$-module. From (\cite{CRTofM}, Exercise 6.7), we find that if $Ass_B M$ is a finite set then so is $Ass_A M$. For arbitrary $B$-module $M$ we have a proof given in the Corollary 1.7 of \cite{Ya}. This will be used in the following theorem.
\end{rem}

\begin{thm}
(a) Let $(R,\m)$ be a Noetherian local ring and $R\rightarrow S$ be a pure extension of Noetherian rings. Then, for Lyubeznik functor $\T$, finiteness of $\Ass_S \T(S)$ implies the finiteness of $\Ass_R \T(R)$.\newline 
In particular, for every $i\geq 0$, if $\Ass_{S} H^i_J(S)$ is a finite set for every ideal $J\subset S$, then for every $i\geq 0$, $\Ass_{R} H^i_{I}(R)$ is also finite set for every ideal $I\subset R$.\newline

(b) Let $R\rightarrow S$ be a pure extension of Noetherian rings (where $R$ is not necessarily a local ring). Then for every $i\geq 0$, if $\Ass_{S} H^i_J(S)$ is a finite set for every ideal $J\subset S$, then for every $i\geq 0$, $\Ass_{R} H^i_{I}(R)$ is also finite set for every ideal $I\subset R$.
\end{thm}

\begin{proof}
(a) For pure extension $R\rightarrow S$, consider the extension $\hat{R}\rightarrow \hat{S}$, where $\hat{S}$ is completion of $S$ in $\m$-adic topology. From Corollary 6.13 of \cite{HR}, we find that $\hat{R}\rightarrow \hat{S}$ is also pure. From Exercise 9.5 in page 64 of \cite{Hu}, we find that extension $\hat{R}\rightarrow \hat{S}$ actually splits. For Lyubeznik functor $\T$, consider $\T(R)$. Since $\Ass_{S} \T(S)$ is a finite set, from faithfully flat ring extension $S\rightarrow \hat{S}$, $\Ass_{\hat{S}} \T(\hat{S})$ is also finite set, see Corollary 1. Since, extension $\hat{R}\rightarrow \hat{S}$ actually splits, using Theorem 1.1 of \cite{Nu} we have the finiteness of $\Ass_{\hat{R}} \T(\hat{R})$. Now, again by Corollary 1, from the result of faithfully flat base change we observe that $\Ass_{R} \T(R)$ is a finite set. Thus we conclude.

Second assertion is immediate.

(b) For pure extension $R\rightarrow S$, using corollary 6.8 of \cite{HR}, $H^i_{I}(R)$ can be thought as $R$-submodule of $H^i_{IS}(S)$. Here $S$ is also an $R$ module, we can think $H^i_{IS}(S)$ as $H^i_{I}(S)$ due to base change. Since $\Ass_{S} H^i_{IS}(S)$ is a finite set, from above Remark $\Ass_{R} H^i_I(S)$ is also a finite set. Since $H^i_{I}(R)\subset H^i_{I}(S)$ as an $R$-submodule, we  get that $\Ass_{R} H^i_I(R)$ is also finite.  
\end{proof}

The result of (b) of Theorem 4.1 is known and it is already obtained in (\cite{Nu}, Proposition 2.2) and the proof of Theorem 4.1 (b) is similar to that of (\cite{Nu}, Proposition 2.2), except the use of Corollary 1.7 of \cite{Ya}, as mentioned in Remark 3. 

There are immediate applications of Proposition 4.1, which are given in the following corollaries and in remark 5. It is to be noted that Corollary 4 and Corollary 5 are known (\cite{Nu}). 

\begin{cor}
Let $G$ be a linearly reductive affine linear algebraic group over a field $k$ of arbitrary characteristic acting $k$ rationally on a regular Noetherian $k$-algebra $S$. Then for ring of invariant $S^G$, for every $i\geq 0$, $\Ass_{S^G}H^i_I(S^G)$ is finite for every ideal $I\subset S^G$.
\end{cor}


\begin{cor}
Let $K$ be a field and let $R$ be a regular domain containing $K$. Let $G$ be a finite subgroup of the group of automorphisms of $R$ and assume that $|G|$ is invertible in $K$. Let $R^G$ be the ring of invariants of $G$. Let $I$ be an ideal in $R^G$. Then for all $i\geq 0$, if $R$ has the property that $\Ass_R H^i_J (R)$ is a finite set for all ideals $J$ of $R$ and for all $i\geq 0$, then $R^G$ has similar property.\newline 
In particular if $K$ has characteristic $p > 0$ or if $R$ is a finitely generated K-algebra then the number of associated primes of $H^i_I (R^G)$ is finite.
\end{cor}


\begin{cor}
Let $S$ be normal domain and $S^G$ be cyclically pure. For every $i\geq 0$, if $\Ass_{S} H^i_J(S)$ is a finite set for every ideal $J\subset S$, then for every $i\geq 0$, $\Ass_{S^G} H^i_{I}(S^G)$ is also finite set for every ideal $I\subset R$.  
\end{cor}


\begin{rem}
Proposition 4.1 implies that certain class of Cohen-Macaulay rings which arise in the following ways enjoy the finiteness property of associated primes of the local cohomologies.
 
(a) From the `Main Theorem' of \cite{HR} we find that $S^G$s of the above corollary are Cohen-Macaulay. The Cohen-Macaulay rings which arise in this fashion, we have finiteness of $\Ass_{S^G}H^i_I(S^G)$ for every $i\geq 0$ and for every ideal $I\subset S^G$. 

(b) Observe the following theorem from \cite{HR}: If $S$ is a regular Noetherian ring of characteristic $p>0$,
and $R$ is a pure subring of $S$, then R is Cohen-Macaulay. The Cohen-Macaulay rings which arise in this fashion, we also have finiteness of $\Ass_{R}H^i_I(R)$ for every $i\geq 0$ and for every ideal $I\subset R$. 

In the above situation, when $R$ is a cyclically pure subring of $S$, see Remark 6.
\end{rem}

We need the following lemma to prove our main results.

\begin{lemma}
Let $R\rightarrow S$ be a cyclically pure ring extension of Noetherian rings, where $R$ is a local ring with maximal ideal $\m$. Then $\hat{R}\rightarrow \bar{S}= S\otimes \hat{R}$ is also cyclically pure.
\end{lemma}

\begin{proof}
We observe the following fact: Consider a ring homomorphism $A\rightarrow B$ of Noetherian rings where $(A,\m)$ is a Noetherian local ring with maximal ideal $\m$. Let for every $\m$-primary ideal $\q\subset A$, $\q B\cap A=\q$, then for every ideal $I$ of $A$, $IB\cap A=I$ i.e. ring homomorphism is cyclically pure. 

To see this we mention the following fact: For any ideal $I\subset A$, $I$ can be written as an arbitrary intersection of $\m$-primary ideals. This is due to Krull-Intersection Theorem. As for example one can think $I=\bigcap^{\infty}_{n=1}(I+{\m}^n)$. Let $I=\bigcap_{i\in \Omega} {\q}_i$ where $\Omega$ is an arbitrary index set. Thus $I\subset IB\cap A\subset (\bigcap_{i\in \Omega} {\q}_i)B\cap A\subset (\bigcap_{i\in \Omega} ({\q}_iB))\cap A=\bigcap_{i\in \Omega} ({\q}_iB\cap A)=\bigcap_{i\in \Omega} {\q}_i=I$. This implies $IB\cap A=I$. 

Since $R\rightarrow S$ is cyclically pure, for every $\m$-primary ideal ${\q}_i$ we have ${\q}_iS\cap R={\q}_i$. Let $\hat{R}$ be the completion of $R$ in $m$-adic topology. Consider the following commutative diagram.  
\[
  \xymatrix
{
  & \hat{R} 
    \ar@{->}[r]
 & S\otimes\hat{R}=\bar{S}  
 \\
  & R
	\ar@{->}[u]
\ar@{->}[r]
 & S
  \ar@{->}[u]      
 }
\]
Going to the ring $\hat{R}$ with a maximal ideal $\hat{\m}$ via faithfully flat ring extension $R\rightarrow \hat{R}$ we find that every $\hat{\m}$-primary ideal $\hat{\q}_i$ is of the form $\hat{\q}_i={\q}_i\hat{R}$. Consider the ring homomorphism $\hat{R}\rightarrow \bar{S}=S\otimes \hat{R}$. Here ${\q}_iS\cap R={\q}_i$ is equivalent to $R/{\q}_i \rightarrow S/{\q}_i S$ is injective. Tensoring with $\hat{R}$ the last map remains injective i.e. we get $\hat{R}/\hat{\q}_i \rightarrow \bar{S}/\hat{{\q}_i}\bar{S}$ is injective. Thus for every $\hat{\m}$-primary ideal we have $\hat{{\q}_i}\hat{S}\cap \hat{R}=\hat{{\q}_i}$. Thus from above paragraph of the proof we get $\hat{R}\rightarrow \bar{S}$ is also a cyclically pure ring extension.
\end{proof}

We recall that a Noetherian semilocal ring is analytically unramified if its completion (by its Jacobson radical) is reduced see \cite{CAofM}. Now, we state another main result of this section, regarding the behaviour of the finiteness of the set of associated primes of local cohomologies under cyclically pure base change. 

\begin{thm}
Consider a cyclically pure ring homomorphism $R\rightarrow S$ of Noetherian rings where $R$ is a analytically unramifed Noetherian local ring with maximal ideal $\m$.\newline 
(a) Then, for Lyubeznik functor $\T$, finiteness of $\Ass_S \T(S)$ implies the finiteness of $\Ass_R \T(R)$.\newline 
(b) In particular, for every $i\geq 0$, if $\Ass_{S} H^i_J(S)$ is a finite set for every ideal $J\subset S$, then for every $i\geq 0$, $\Ass_{R} H^i_{I}(R)$ is also finite set for every ideal $I\subset R$.  
\end{thm}

\begin{proof}
(a) From Lemma 4.2, we find that $\hat{R}\rightarrow \bar{S}$ is also cyclically pure. Since $\hat{R}$ is also reduced, using \cite{Ho1} or from above Remark 2, $\hat{R}\rightarrow \bar{S}$ is actually a pure extension. Moreover, from Exercise 9.5 in page 64 of \cite{Hu}, we find that extension $\hat{R}\rightarrow \bar{S}$ actually splits. For Lyubeznik functor $\T$, consider $\T(R)$. Since $\Ass_{S} \T(S)$ is a finite set, from faithfully flat ring extension $S\rightarrow \bar{S}$, $\Ass_{\bar{S}} \T(\bar{S})$ is also finite set, see Corollary 1. Since, extension $\hat{R}\rightarrow \bar{S}$ actually splits, using Theorem 1.1 of \cite{Nu} we have the finiteness of $\Ass_{\hat{R}} \T(\hat{R})$. Now, again by Corollary 1, from the result of faithfully flat base change we observe that $\Ass_{R} \T(R)$ is a finite set. Thus we conclude.

(b) This assertion is immediate from (a). 
\end{proof}

We observe the following corollary.

\begin{cor}
Let $(S,\n)$ be a regular local ring which is either (a) of equicharacteristic or (b) unramified. Let $(R,\m)$ be its cyclically pure subring. Then for Lyubeznik functor $\T$, $\Ass_R \T(R)$ is a finite set.\newline 
In particular, for every $i\geq 0$, $\Ass_{R} H^i_{I}(R)$ is finite set for every ideal $I\subset R$. 
\end{cor}

\begin{proof}
Let $\hat{S}$ and $\hat{R}$ be the completions of $S$ and $R$ in their respective $\n$-adic and $\m$-adic topology. Here $R\rightarrow S$ is cyclically pure. From Lemma 6.7 of \cite{AS} we find that $\hat{R}\rightarrow \hat{S}$ is also cyclically pure and hence injective. This suggests that the cyclically pure subring $R$ of $S$ is analytically unramified since $\hat{S}$ as well as $S$ are domain. From \cite{Ly1}, \cite{Ly2} and \cite{Ly3} we find that $\Ass_S \T(S)$ as well as $\Ass_{S} H^i_J(S)$ is a finite set, for every $i\geq 0$, for every ideal $J\subset S$. Thus using Theorem 4.3, we can conclude. 
\end{proof}

\begin{rem}
Theorem 6.8 of \cite{AS}, which is a generalization of the theorem of \cite{HR} (as mentioned in (b) of Remark 5), states that a cyclically pure subring of a regular local ring is Cohen-Macaulay. Thus Corollary 7 implies that a class of Cohen-Macaulay rings which arise as cyclically pure subrings of a regular local rings, enjoy the finiteness property of the set of associated primes of the local cohomologies.
\end{rem}

We conclude the section with following observation.

\begin{rem}
In Lemma 4.2, we show that for a cyclically pure ring homomorphism $R\rightarrow S$ of Noetherian rings where $R$ is a Noetherian local ring with maximal ideal $\m$, $\hat{R}\rightarrow \bar{S}= S\otimes\hat{R}$ is also cyclically pure. It is to be noted that $\bar{S}$ will be equal to $\hat{S}$, if $S$ is module finite over $R$, where $\hat{S}$ is completion of $S$ in $\m$-adic topology. In Lemma 6.7 of \cite{AS}, we also observe that for Noetherian local rings $(S,\n)$ and $(R,\m)$, if $R\rightarrow S$ is cyclically pure then so is $\hat{R}\rightarrow \hat{S}$. But, there $\hat{S}$ and $\hat{R}$ are the completions of $S$ and $R$ in their respective $\n$-adic and $\m$-adic topologies. 
\end{rem}
\section{appendix: behaviour of the bass numbers under pure and cyclically pure ring extensions}

In this section, we prove results regarding the behaviour of Bass numbers under pure and cyclically pure ring extensions (see Theorem 5.2 and Theorem 5.3). At first, we recall the definition of the Bass number: Let $T$ be a module over a ring $R$ and $\p$ be a prime ideal of $R$. We define the $j$-th Bass number $\mu_j(\p,T)$ at $\p$ to be $\mu_j(\p,T)= dim_{\kappa{(\p)}} \Ext_{R_{\p}}^j (\kappa{(\p)},T_{\p})$, where $\kappa{(\p)} = R_{\p} /{\p}R_{\p}$. 

We observe the following result regarding the Bass number.

\begin{prop}
Consider Noetherian local ring $(R,\m)$ and let $\hat{R}$ be the $\m$-adic completion of $R$. Then for every $R$-module $E$, all the Bass numbers of $E$ is finite if and only if all the Bass numbers of $E\otimes_{R} \hat{R}$ is finite.
\end{prop}

\begin{proof}
From definition of the Bass numbers, we only prove the proposition for maximal ideal $\m$ of $R$. Since $R\rightarrow \hat{R}$ is faithfully flat, we have an injection $Ext_{R}^j (R/\m,E)\rightarrow Ext_{\hat{R}}^j (\hat{R}/\m\hat{R},E\otimes_{R} \hat{R})$. Set $k=R/\m = \hat{R}/\m\hat{R}$. If $\mu_j(\m\hat{R},E\otimes_R \hat{R})$ is finite, then $Ext_{\hat{R}}^j (\hat{R}/\m\hat{R},E\otimes_{R} \hat{R})$ is also finite as an $R$-module. Since $Ext_{R}^j (R/\m,E)$ is an $R$-submodule of $Ext_{\hat{R}}^j (\hat{R}/\m\hat{R},E\otimes_{R} \hat{R})$, it is also finitely generated as $R$- module and we get the finiteness of $\mu_j(\m,E)$. Other implication is immediate since finite generated module remains finitely generated under base change. This proves the proposition.
\end{proof}

We state the first result of this section.

\begin{thm}
Let $(R,\m)$ and $(S,\n)$ be Noetherian local rings and $R\rightarrow S$ be a pure ring extension, such that $S$ be a finitely generated as $R$-module. Let $S$ be Cohen-Macaulay.\newline 
(a) Then for Lyubeznik functor $\T$, finiteness of all the Bass numbers of $\T(S)$ implies the finiteness of all the Bass numbers of $\T(R)$.\newline 
(b) In particular, for some ideal $I\subset R$, if all the Bass numbers of $H^i_{IS}(S)$ are finite, then all the Bass numbers of $H^i_{I}(R)$ are also finite.\newline
\end{thm}

\begin{proof}
(a) For pure extension $R\rightarrow S$, consider the extension $\hat{R}\rightarrow \hat{S}$, where $\hat{S}$ is completion of $S$ in $\m$-adic topology. Since $S$ is finitely generated as an $R$-module, $\hat{S}$ is also completion of $S$ in $\n$-adic topology. From Corollary 6.13 of \cite{HR}, we find that $\hat{R}\rightarrow \hat{S}$ is also pure. From Exercise 9.5 in page 64 of \cite{Hu}, we find that extension $\hat{R}\rightarrow \hat{S}$ actually splits. For Lyubeznik functor $\T$, consider $\T(R)$. Since all the Bass number $\T(S)$ are finite, from faithfully flat ring extension $S\rightarrow \hat{S}$, all the Bass numbers of $\T(\hat{S})$ are also finite, see Proposition 5.1. Now, $\hat{S}$ is Cohen-Macaulay ring, which is also finitely  generated $\hat{R}$-module and moreover extension $\hat{R}\rightarrow \hat{S}$ actually splits. Thus, using Theorem 1.2 of \cite{Nu}, we have the finiteness of all the Bass numbers of $\T(\hat{R})$. Now, again by Proposition 5.1, from the result of faithfully flat base change we observe that all the Bass numbers of $\T(R)$ are finite. Thus we conclude.

(b) This assertion is immediate from (a).
\end{proof}

Now, we state the second main result of this section.

\begin{thm}
Consider a cyclically pure ring homomorphism $R\rightarrow S$ of Noetherian local rings where $(R,\m)$ is a analytically unramifed and $(S,\n)$ is Cohen-Macaulay. Let $S$ be a finitely generated $R$-module.\newline
(a) Then, for Lyubeznik functor $\T$, finiteness of all the Bass numbers of $\T(S)$ implies the finiteness of of all the Bass numbers of $\T(R)$.\newline 
(b) In particular, for some ideal $I\subset R$, if all the Bass numbers of $H^i_{IS}(S)$ are finite, then all the Bass numbers of $H^i_{I}(R)$ are also finite.
\end{thm}

\begin{proof}
(a) From Lemma 4.2, we find that $\hat{R}\rightarrow \bar{S}$ is also cyclically pure. Since $S$ is finitely generated $R$-module, $\bar{S}= \hat{S}$, where latter is $\n$-adic completion of $S$. Since $\hat{R}$ is also reduced, using \cite{Ho1} or from above Remark 2, $\hat{R}\rightarrow \hat{S}$ is actually a pure extension. Moreover, from Exercise 9.5 in page 64 of \cite{Hu}, we find that extension $\hat{R}\rightarrow \hat{S}$ actually splits. For Lyubeznik functor $\T$, consider $\T(R)$. Since all the Bass numbers of $\T(S)$ are finite, from faithfully flat ring extension $S\rightarrow \hat{S}$, all the Bass numbers of $\T(\hat{S})$ are also finite, see Proposition 5.1. Now, $\hat{S}$ is Cohen-Macaulay ring, which is also finitely  generated $\hat{R}$-module and moreover extension $\hat{R}\rightarrow \hat{S}$ actually splits. Thus, using Theorem 1.2 of \cite{Nu}, we have the finiteness of all the Bass numbers of $\T(\hat{R})$. Now, again by Proposition 5.1, from the result of faithfully flat base change we observe that all the Bass numbers of $\T(R)$ are finite. Thus we conclude.

(b) This assertion is immediate from (a). 
\end{proof}

{\textbf{Acknowledgement:}}\newline

I would like to thank Luis N$\acute{u}$$\tilde{n}$ez-Betancourt for his valuable comments and suggestions for earlier version (arXiv: 1512.05672v1) of this paper. I also wish to thank Tony J. Puthenpurakal for his comments on the results of the previous version this paper. In this version an Appendix on `Behaviour of the Bass numbers under pure and cyclically pure Ring Extensions' is added. 


\begin{thebibliography}{}
\bibitem[AS]{AS}M. Aschenbrenner and H. Schouten, \emph{Lefschetz extension, tight closure and big Cohen-Macaulay algebras}, Israel Journal of Mathematics, 161 (2007), 221-310.
\bibitem[BBLSZ]{BBLSZ}B. Bhatt, M. Blickle, G. Lyubeznik, A. Singh and W. Zhang, \emph{Local Cohomology Modules of a Smooth $\ZZ$-Algebra have Finitely Many Associated Primes}, Inventiones Mathematicae, 197 (2014) 509-519, arXiv: 1304.4692v2 [math.AC].
\bibitem[BH]{BH}W. Bruns and J. Herzog, \emph{Cohen-Macaulay Rings}, Cambridge University Press, 1997.
\bibitem[Gl]{Gl}S. Glaz, \emph{Commutative coherent rings}, Springer LNM 1371, Springer Verlag, 1989.
\bibitem[Ho]{Ho1} Hochster, \emph{Cyclic Purity vs Purity in Excellent Noetherian Rings}, Transaction of American Mathematical Society, Vol. 231, No. 2, 463-488, 1977.
\bibitem[HR]{HR} Hochster and J. Roberts, \emph{Rings of Invariants of Reductive Groups Acting on Regular Rings are Cohen-Macaulay}, Advanced in Mathematics 13, 115-175, 1974. 
\bibitem[HS]{HS} Huneke and Sharp, \emph{Bass Numbers of Local Cohomology Modules}, Transaction of American Mathematical Society, Vol. 339, No. 2, 765-779, 1993.
\bibitem[Hu1]{Hu}C. Huneke, \emph{Problems on local cohomology, in: Free resolutions in commutative algebra and algebraic
geometry (Sundance, Utah, 1990)}, 93-108, Res. Notes Math. 2, Jones and Bartlett, Boston, MA, 1992.
\bibitem[Hu2]{Hu2}C. Huneke, \emph{Tight closure and its application { With} an appendix by {Melvin} {Hochster}}, CBMS Reg. Conf. Ser. in Math., vol.~88, American Mathematical Society, Providence, RI, 1996., 9308, Res. Notes Math. 2, Jones and Bartlett, Boston, MA, 1992.
\bibitem[Ka]{Ka}M. Katzman, \emph{An example of an infinite set of associated primes of a local cohomology module}, J. Algebra, 252, no. 1, 161?66, 2002.
\bibitem[Ly1]{Ly1}G. Lyubeznik, \emph{Finiteness Properties of Local Cohomology modules (an application of D-modules to commutative algebra)}, Invent. Math. 113, 41-55, 1993.
\bibitem[Ly2]{Ly2}G. Lyubeznik, \emph{Finiteness Properties of Local Cohomology Modules for Regular Local Rings of Mixed Characteristic: the Unramified Case}, Comm. Alg. 28, 5867-5882, 2000.
\bibitem[Ly3]{Ly3}G. Lyubeznik, \emph{$F$-modules: applications to Local Cohomology and $D$-modules in Characteristic $p>0$}, J Reine Angew. Math. 491, 65-130, 1997.
\bibitem[Mar]{Mar}T. Marley, \emph{The associated primes of local cohomology modules over ring of small dimension}, Manuscripta Mathematica 104, 519-525, 2001.
\bibitem[Ma1]{CAofM}H. Matsumura, \emph{Commutative Algebra}, The Benjamin/Cummings Publishing Company, 1980.
\bibitem[Ma2]{CRTofM}H. Matsumura,\emph{Commutative Ring Theory}, Cambridge University Press, 1990. 
\bibitem[Nu1]{Nu}L. N$\acute{u}$$\tilde{n}$ez-Betancourt \emph{Local cohomology properties of direct summands}, J. Pure Appl. Algebra 216, no. 10, 2137-2140, 2012.
\bibitem[Nu2]{Nu2}L. N$\acute{u}$$\tilde{n}$ez-Betancourt \emph{Local cohomology modules of polynomial and power series rings over rings of small dimensions}, Illinois Journal of Mathematics, Volume 57, No. 1, Pages 279-294, 2013, arXiv: 1207.1896v2 [math.AC].
\bibitem[Nu3]{Nu3}L. N$\acute{u}$$\tilde{n}$ez-Betancourt \emph{On certain rings of differentiable type and finiteness properties of local cohomology}, Journal of Algebra, Volume 379, April 2013, Pages 1-10, arXiv: 1202.3597 [math.AC].
\bibitem[Pu]{Pu2}T. J. Puthenpurakal, \emph{{Associated primes of local cohomology modules over regular rings}}, arXiv: 1411.2734 [math.AC]. 
\bibitem[Si]{Si}A. K. Singh, \emph{p-torsion elements in local cohomology modules}, Math. Res. Lett. 7, 165-176, 2000.
\bibitem[SS]{SS}A. K. Singh and I. Swanson, \emph{Associated primes of local cohomology modules and of Frobenius powers}, Int. Math. Res. Not., no. 33, 1703-1733, 2004.
\bibitem[Sc1]{Sc1} H. Schoutens, \emph{Canonical big Cohen-Macaulay algebras and rational singularities}, Illinois J. Math. 41, 131-150, 2004.
\bibitem[Sc2]{Sc2} H. Schoutens, \emph{On the vanishing of Tor for the absolute integral closure},  J. Algebra 275, 567-574, 2004. 
\bibitem[Sc3]{Sc}H. Schoutens, \emph{Pure subrings of regular rings are pseudo-rational}, Transaction of American Mathematical Society, Vol. 360, No. 2, 609-627, 2008. 
\bibitem[Ya]{Ya}S. Yassemi, \emph{Weakly associated primes under change of rings}, Communications in Algebra, 26(6), 2007-2018, 1998.
\end{thebibliography}
\end{document}